\documentclass[a4paper,12pt]{article}
\usepackage{amsmath,amsfonts,amssymb,amscd}
\usepackage{indentfirst,graphicx,epsfig}
\usepackage{graphicx,psfrag}
\usepackage{amsthm, amsfonts, color}
\usepackage[bookmarksnumbered, plainpages]{hyperref}
\input{epsf}

\title{Skew Randi\'c Matrix and Skew Randi\'c Energy\footnote{Supported by NSFC No.11371205 and
PCSIRT.} }
\author{\small{ Ran Gu, Fei Huang, Xueliang Li }\\
{\small  Center for Combinatorics and LPMC-TJKLC}\\
{\small Nankai University, Tianjin 300071, P.R. China}\\
{\small Email: guran323@163.com, huangfei06@126.com, lxl@nankai.edu.cn}}
\date{}

\begin{document}

\makeatletter
  \newcommand\figcaption{\def\@captype{figure}\caption}
  \newcommand\tabcaption{\def\@captype{table}\caption}
\makeatother
\newtheorem{pro}{Proposition}[section]
\newtheorem{defn}{Definition}[section]
\newtheorem{theorem}{Theorem}[section]
\newtheorem{lemma}[theorem]{Lemma}
\newtheorem{coro}[theorem]{Corollary}

\maketitle

\begin{abstract}
Let $G$ be a simple graph with an orientation $\sigma$, which
assigns to each edge a direction so that $G^\sigma$ becomes a
directed graph. $G$ is said to be the underlying graph of the
directed graph $G^\sigma$. In this paper, we define a weighted skew
adjacency matrix with Rand\'c weight, the skew Randi\'c matrix ${\bf
R_S}(G^\sigma)$, of $G^\sigma$ as the real skew symmetric matrix
$[(r_s)_{ij}]$ where $(r_s)_{ij} = (d_id_j)^{-\frac{1}{2}}$ and
$(r_s)_{ji} = -(d_id_j)^{-\frac{1}{2}}$ if $v_i \rightarrow v_j$ is
an arc of $G^\sigma$, otherwise $(r_s)_{ij} = (r_s)_{ji} = 0$. We
derive some properties of the skew Randi\'c energy of an oriented
graph. Most properties are similar to those for the skew energy of
oriented graphs. But, surprisingly, the extremal oriented graphs
with maximum or minimum skew Randi\'c energy are completely
different.
\\[2mm]



\end{abstract}

\section{Introduction}

In this paper we are concerned with simple finite graphs. Undefined
notation and terminology can be found in \cite{Bondy}. Let $G$ be a
simple graph with vertex set $V(G) = \{v_1, v_2,\ldots, v_n\}$, and
let $d_i$ be the degree of vertex $v_i$, $i = 1, 2, \cdots, n$. We
use $P_k$ to denote the path on $k$ vertices, and the \emph{length}
of a path is the number of edges that the path uses.

The Randi\'c index \cite{Randic} of $G$ is defined as the sum of
$\frac{1}{\sqrt{d_id_j}}$ over all edges $v_iv_j$ of $G$. This
topological index  was first proposed by Randi\'c \cite{Randic} in
1975 under the name ``branching index". In 1998, Bollob\'as and
Erd\H{o}s \cite{BE} generalized this index as $R_\alpha =
R_\alpha(G) =\sum _{i\sim j}(d_id_j)^\alpha$, called general
Randi\'c index.

Let $\textbf{A}(G)$ be the $(0,1)$-adjacency matrix of $G$. The
\emph{spectrum }$Sp(G)$ of $G$ is defined as the spectrum of
$\mathbf{A}(G)$. The Randi\'c matrix \cite{Gutman}
$\textbf{R}=\textbf{R}(G)$ of order $n$ can be viewed as a weighted
adjacency matrix, whose $(i,j)$-entry is defined as
\begin{equation*}
r_{ij}=
\left\{
  \begin{array}{ll}
   0 & \hbox{ if $i$ = $j$,} \\
    (d_id_j)^{-\frac{1}{2}} & \hbox{ if the vertices $v_i$ and $v_j$ of $G$ are adjacent}, \\
    0 & \hbox{ if the vertices $v_i$ and $v_j$ of $G$ are not adjacent.}
  \end{array}
\right.
\end{equation*}
The polynomial $\varphi_R(G,\lambda)=det(\lambda
\textbf{I}_n-\textbf{R})$ will be referred to as the
$R-$characteristic polynomial of $G$.  Here and later by
$\textbf{I}_n$ is denoted the unit matrix of order $n$.

The spectrum $Sp_{\textbf{R}}(G^\sigma)$ of $G$ is defined as the
spectrum of $\textbf{R}(G)$.  Denote the spectrum
$Sp_{\textbf{R}}(G^\sigma)$ of $G$ by $\{\lambda_1,\lambda_2,\ldots
,\lambda_n\}$ and label them in non-increasing order. The energy of
the Randi\'c matrix is defined as $RE=RE(G)=\sum_{i=1}^n|\lambda_i|$
and is called Randi\'c energy. There are other kinds of Randi\'c
type matrices and energy, for details  see \cite{GHL}, \cite{GHL2}.

Let $G$ be a simple graph with an orientation $\sigma$, which
assigns to each edge a direction so that $G^\sigma$ becomes a
directed graph. $G$ is said to be the underlying graph of the
directed graph $G^\sigma$. With respect to a labeling, the
skew-adjacency matrix $\textbf{S}(G^\sigma)$ is the real skew
symmetric matrix $[s_{ij} ]$ where $s_{ij} = 1$ and $s_{ji} = -1$ if
$v_i \rightarrow v_j$ is an arc of $G^\sigma$, otherwise $s_{ij} =
s_{ji} = 0$.

Now we define the {\it skew Randi\'c matrix}
$\textbf{R}_s=\textbf{R}_s(G^\sigma)$ of order $n$, whose
$(i,j)$-entry is
\begin{equation*}
(r_s)_{ij}=
\left\{
  \begin{array}{ll}
   (d_id_j)^{-\frac{1}{2}} & \hbox{ if $v_i \rightarrow v_j$,} \\
    -(d_id_j)^{-\frac{1}{2}} & \hbox{ if $v_j \rightarrow v_i$}, \\
    0 & \hbox{ Otherwise.}
  \end{array}
\right.
\end{equation*}

If $G$ does not possess isolated vertices, and $\sigma$ is an
orientation of $G$, then it is easy to check that
$$\textbf{R}_s(G^\sigma)=\textbf{D}^{-\frac{1}{2}}\textbf{S}(G^\sigma)
\textbf{D}^{-\frac{1}{2}},$$ where $\textbf{D}$ is the diagonal matrix of
vertex degrees.

The polynomial $\varphi_{R_s}(G,\lambda)= \mathrm{det}(\lambda
\textbf{I}_n-\textbf{R}_s)$ will be referred to as the
$R_s$-characteristic polynomial of $G^\sigma$. It is obvious that
$\textbf{R}_s(G^\sigma)$ is a real skew symmetric matrix. Hence the
eigenvalues $\{\rho_1,\rho_2,\ldots ,\rho_n\}$ of
$\textbf{R}_s(G^\sigma)$ are all purely imaginary numbers. The skew
Randi\'c spectrum $Sp_{\textbf{R}_s}(G^\sigma)$ of $G^\sigma$ is
defined as the spectrum of $\textbf{R}_s(G^\sigma)$.

The energy of $\textbf{R}_s(G^\sigma)$, called {\it skew Randi\'c
energy} which is defined as the sum of its singular values, is the
sum of the absolute values of its eigenvalues. If we denote the skew
Randi\'c energy of $G^\sigma$ by $RE_s(G^\sigma)$, then
$RE_s(G^\sigma) = \sum_{i=1}^n|\rho_i|$.

Note that the skew Randi\'c matrix ${\bf R_S}(G^\sigma)$ is a
weighted skew-adjacency matrix of $G^\sigma$ with the Rand\'c
weight. In this paper, we derive some properties of the skew
Randi\'c energy of an oriented graph. Most properties are similar to
those for the skew energy of an unweighted oriented graph. But,
surprisingly, the extremal oriented graphs with maximum or minimum
skew Randi\'c energy are completely different.

\section{Basic properties }

The following proposition on the skew Randi\'c spectra of oriented graphs is obvious.
\begin{pro}\label{pro1}
Let $\{i\mu_1, i\mu_2, \ldots, i\mu_n\}$ be the skew Randi\'c
spectrum of $G^\sigma$, where $\mu_1\geq \mu_2\geq \ldots\geq
\mu_n$. Then (1) $\mu_j = -\mu_{n+1-j}$ for all $1 \leq j \leq n$;
(2) when $n$ is odd, $\mu_{(n+1)/2} = 0$ and when $n$ is even,
$\mu_{n/2}\geq 0$; and (3) $\sum_{i=1}^n \mu_i^2=2R_{-1}(G)$, where
$R_{-1}(G)$ is the general Randi\'c index of $G$ with $\alpha =-1$.
\end{pro}

Let $G$ be a graph. A \emph{linear subgraph} $L$ of $G$ is a
disjoint union of some edges and some cycles in $G$. Let
$\mathcal{L}_i(G)$  be the set of all linear subgraphs $L$ of $G$
with $i$ vertices. For a  linear subgraphs $L\in \mathcal{L}_i(G)$,
denote by $p_1(L)$ the number of components of size $2$ in $L$ and
$p_2(L)$ the number of cycles in $L$.

 Let \begin{equation}\label{1}
 \varphi_R(G,\lambda)=a_0\lambda^n+a_1\lambda^{n-1}+\dots+a_{n-1}\lambda+a_n \end{equation}
 be the $R-$characteristic polynomial of $G$. From \cite{DM}, we know that
\begin{equation}\label{ai}
a_i=\sum_{L\in \mathcal{L}_i}(-1)^{p_1(L)}(-2)^{p_2(L)}W(L),
\end{equation}
where $W(L)=\prod_{v\in V(L)} \frac{1}{d(v)}.$ If $G$ is bipartite, then $a_i=0$ for all odd $i$.

Now considering the oriented graph $G^\sigma$, let $C$ be an even
cycle of $G$. We say $C$ is \emph{evenly oriented relative to}
$G^\sigma$ if it has an even number of edges oriented in the
direction of the routing. Otherwise $C$ is oddly oriented.

We call a linear subgraph $L$ of $G$ \emph{evenly linear} if $L$
contains no cycle with odd length and denote by $\mathcal{EL}_i(G)$
(or $\mathcal{EL}_i$ for short) the set of all evenly linear
subgraphs of $G$ with $i$ vertices. For an evenly linear subgraph
$L\in \mathcal{EL}_i(G)$, we use $p_e(L)$ (resp., $p_o(L)$)  to
denote the number of evenly (resp., oddly) oriented  cycles  in $L$
relative to $G^\sigma$.

Consider $G^\sigma$ as a weighted oriented graph with each edge
$v_iv_j$ assigned the weight $\frac{1}{d(v_i)d(v_j)}$. Then the skew
Randi\'c characteristic polynomial of $G^\sigma$ equals to the skew
characteristic polynomial of weighted oriented graph $G^\sigma$.  In
\cite{GX}, the authors studied the skew characteristic polynomial of
weighted oriented graph. From their results, we can derive the skew
Randi\'c characteristic polynomial of an oriented graph $G^\sigma$
as follows.

\begin{theorem}
Let
 \begin{equation}\label{2}
 \varphi_{R_s}(G^\sigma,\lambda)=
\mathrm{det}(\lambda \textbf{I}_n-\textbf{R}_s) = c_0\lambda^n + c_1\lambda^{n-1} +\ldots+c_{n-1}\lambda+c_n
 \end{equation}
 be the $R_s$-characteristic polynomial of $G^\sigma$. Then
\begin{equation}\label{ci}
c_i=\sum_{L\in \mathcal{EL}_i}(-2)^{p_e(L)}2^{p_o(L)}W(L).
\end{equation}
In particular, we have
(i) $c_0 = 1$, (ii) $c_2 = R_{-1}(G)$, the general Randi\'c index with $\alpha =-1$ and (iii) $c_i = 0$ for all odd $i$.
\end{theorem}

\section{The upper and lower bounds}

\begin{theorem}\label{THB}
$\sqrt{4R_{-1}(G)+n(n-2)p^{\frac{2}{n}}}\leq RE_s(G^\sigma)\leq 2 \sqrt{\lfloor\frac{n}{2}\rfloor R_{-1}(G)}$, where $p=|det\textbf{R}_s|=\prod_{i=1}^n|\rho_i|.$

\end{theorem}
\begin{proof}
Let $\{i\mu_1, i\mu_2, \ldots, i\mu_n\}$ be the skew Randi\'c spectrum of $G^\sigma$, where $\mu_1\geq \mu_2\geq \ldots\geq \mu_n$. Since
 $\sum_{j=1}^n(i\mu_j)^2=tr(\textbf{R}_s^2)= \sum_{j=1}^n\sum_{k=1}^n(r_s)_{jk}(r_s)_{kj}=
-\sum_{j=1}^n\sum_{k=1}^n(r_s)_{jk}^2=-2R_{-1}(G)$,
 we have
$\sum_{j=1}^n|\mu_j|^2=2R_{-1}(G)$.

By Proposition \ref{pro1}, we know that
$\sum_{j=1}^{\lfloor\frac{n}{2}\rfloor}|\mu_j|^2=R_{-1}(G)$ and
$RE_s(G^\sigma)=2\sum_{j=1}^{\lfloor\frac{n}{2}\rfloor}|\mu_i|$.
Applying the Cauchy-Schwartz inequality we have that
\begin{equation}\label{UB}
RE_s(G^\sigma) = 2\sum_{j=1}^{\lfloor\frac{n}{2}\rfloor}|\mu_j|\leq 2\sqrt{\sum_{j=1}^{\lfloor\frac{n}{2}\rfloor}|\mu_j|^2}\sqrt{\lfloor\frac{n}{2}\rfloor}
=2 \sqrt{\lfloor\frac{n}{2}\rfloor R_{-1}(G)}.
\end{equation}
By Proposition \ref{pro1}, we know that
$$[RE_s(G^\sigma)] ^2=\left(2\sum_{j=1}^{\lfloor\frac{n}{2}\rfloor}|\mu_j|\right)^2=4\sum_{j=1}^{\lfloor\frac{n}{2}\rfloor}|\mu_i|^2
+4\sum_{1\leq i\neq j\leq \lfloor\frac{n}{2}\rfloor}|\mu_i||\mu_j|.$$
If $n$ is odd, $p=0$ and $[RE_s(G^\sigma)] ^2\geq4\sum_{j=1}^{\lfloor\frac{n}{2}\rfloor}|\mu_j|^2
=4R_{-1}(G).$ If $n$ is even, by using  arithmetic geometric average inequality, one can get that
$$[RE_s(G^\sigma)] ^2=4\sum_{j=1}^{\lfloor\frac{n}{2}\rfloor}|\mu_j|^2
+4\sum_{1\leq i\neq j\leq \lfloor\frac{n}{2}\rfloor}|\mu_i||\mu_j|\geq 4R_{-1}(G)+ n(n-2)p^{\frac{2}{n}}. $$

Therefore we can obtain the lower bound on  skew Randi\'c energy,
\begin{equation}\label{LB}
RE_s(G^\sigma)\geq \sqrt{4R_{-1}(G)+n(n-2)p^{\frac{2}{n}}}.
\end{equation}
\end{proof}
Note that there are plenty results on the upper and lower bounds on $R_{-1}(G)$,
combining with Theorem \ref{THB}, we can get the upper and lower bounds on  skew Randi\'c energy without the  parameter $R_{-1}(G)$.

Li and Yang \cite{LY} provided the following bounds on $R_{-1}(G)$ given strictly in terms of the order of $G$.

\begin{theorem}\label{t1}
Let $G$ be a graph of order $n$ with no isolated vertices. Then
$$\frac{n}{2(n-1)}\leq R_{-1}(G)\leq \lfloor\frac{n}{2}\rfloor$$
with equality in the lower bound if and only if $G$ is a complete graph, and equality in the
upper bound if and only if either
(i) $n$ is even and $G$ is the disjoint union of $n=2$ paths of length 1, or
(ii) $n$ is odd and $G$ is the disjoint union of $\frac{n-3}{2}$ paths of length 1 and one path of
length 2.
\end{theorem}

To depict the extremal oriented graphs attaining the bounds on skew
Randi\'c energy, we need another result proved by Li and Wang
\cite{LW}. Note that it has been proved that $\lambda_1=1$ is the
largest Randi\'c eigenvalues with the Perron-Frobenius vector
$\alpha ^T = ( \sqrt{d_1},\ldots,\sqrt{d_n})$; see \cite{Gutman}.
\begin{theorem}\label{th0}
Let $G$ be a connected graph with order $n\geq 3$ and size $m$. Let $\alpha ^T = ( \sqrt{d_1},\ldots,\sqrt{d_n})$. Then $G$ has exactly
$k$ $(2\leq k \leq n)$ and distinct Randi\'c eigenvalues if and only if there are $k-1$ distinct none-one real
numbers $\lambda_2,\ldots,\lambda_k$ satisfying
(i) $\textbf{R}(G)- \lambda_k\textbf{I}_n$ is a singular matrix for $2\leq i \leq k$.
(ii) $\prod\limits_{i = 2}^k {{{(\textbf{R}(G) - }}} {\lambda_i}{\textbf{I}_n}) = \frac{{\prod\limits_{i = 2}^k {{{(1 - }}} {\lambda_i})}}{{2m}}\alpha {\alpha ^T}.$
Moreover, 1, $\lambda_2,\ldots,\lambda_k$  are exactly the $k$ distinct Randi\'c eigenvalues of $G$.
\end{theorem}
From the above theorem, the authors  gave the following corollary in \cite{LW}.
\begin{coro}\label{t0}
A connected graph $G$ has exactly two and distinct Randi\'c eigenvalues if and only
if $G$ is a complete graph with order at least two.
\end{coro}

Hence, we can obtain the following result.
\begin{theorem}\label{t1.1}
Let $G^\sigma$ be an oriented graph of order $n$ with no isolated
vertices. Then
$$\sqrt{\frac{2n}{n-1}+n(n-2)p^{\frac{2}{n}}}\leq RE_s(G^\sigma)\leq 2\lfloor\frac{n}{2}\rfloor,$$
where $p=|det\textbf{R}_s(G^\sigma)|$.  The equality in the lower
bound holds if and only if $G$ is a complete graph with exactly two
nonzero skew Randi\'c eigenvalues when $n$ is odd, and
$\textbf{R}_s(G^\sigma)^T\textbf{R}_s(G^\sigma)=\frac{1}{n-1}\mathbf{I_n}$
when $n$ is even. The equality in the upper bound holds if and only
if either $n$ is even and $G$ is the disjoint union of  paths of
length 1, or, $n$ is odd and $G$ is the disjoint union of
$\frac{n-3}{2}$ paths of length 1 and one path of length 2, and
$\sigma$ is an arbitrary orientation of $G$.
\end{theorem}

\begin{proof}
The bounds on $RE_s(G^\sigma)$ comes directly from Theorem \ref{THB}
and Theorem \ref{t1}. We focus on the equality. From Theorem
\ref{t1} and  Theorem \ref{THB},  we know that, if $n$ is odd, the
equality in the lower bound holds  if and only if $G$ is a complete
graph and
$[RE_s(G^\sigma)]^2=4\sum_{i=1}^{\lfloor\frac{n}{2}\rfloor}|\mu_i|^2$,
that is, $\mu_i=0$ for all $i=2,\cdots, \lfloor\frac{n}{2}\rfloor$.
If $n$ is even,  the equality in the lower bound holds  if and only
if $G$ is a complete graph and $|\mu_i|=|\mu_j|$ for all $i\neq j$.
Thus, a complete oriented graph with odd vertices reaches the lower
bound if and only if it has exactly two nonzero skew Randi\'c
eigenvalues or all the skew Randi\'c eigenvalues are zero. Since we
assume $G$ has no isolated vertices, the latter case can not happen.
For an even $n$, the equality in the lower bound  holds  if and only
if $G$ is a  complete graph and there exists a constant $k$ such
that $|\rho_i|^2=k$ for all $i$, which  holds if and only if
$\textbf{R}_s(G^\sigma)^T\textbf{R}_s(G^\sigma)=\frac{1}{n-1}\mathbf{I_n}$.

From Theorem \ref{t1} and  Theorem \ref{THB}, the equality in the
upper bound holds  if and only if $G$ is the graph described in
Theorem \ref{t1} and  $|\rho_i|=|\rho_j|$ for all $1\leq i\neq j\leq
\lfloor\frac{n}{2}\rfloor$, that is,  either $n$ is even and $G$ is
the disjoint union of  paths of length 1, or $n$ is odd and $G$ is
the disjoint union of $\frac{n-3}{2}$ paths of length 1 and one path
of length 2. In both cases,  let $\sigma$ be  an arbitrary
orientation of $G$. By Corollary \ref{cor1} and Corollary \ref{t0},
$G^\sigma$ attains the upper bound, and the converse also holds.
\end{proof}

Now let us consider the bounds on skew Randi\'c energies of trees.
On the index $R_{-1}(T)$ when $T$ is a tree, Clark and Moon
\cite{CM} gave the following result.

\begin{theorem}\label{t2}
For a tree $T$ of order $n$, $1 \leq  R_{-1}(T)\leq
\frac{5n+8}{18}$. The equality  in lower bound holds if and only if
$T$ is the star.
\end{theorem}

Pavlovi\'c, Stojanvoi\'c and Li \cite{PSL} determined the sharp
upper bound on the Randi\'c index $R_{-1}(T)$ among all trees of
order $n$ for every $n \geq 720$.
\begin{theorem}\label{t3}
Let ${T_t^n}$
be a tree of order $n \geq 720$ and $n -1 \equiv t (mod $ $7)$. Denote by ${R^*_{-1}}$ the
maximum value of the Randi\'c index $R_{-1}(T)$ among all trees ${T_t^n}$. Then,
\begin{equation*}
{R^*_{-1}}=
\left\{
  \begin{array}{ll}
   \frac{15n-1}{56} & \hbox{$t = 0$,} \\
    \frac{15n-1}{56}-\frac{1}{56}+\frac{7}{4(n+5)}& \hbox{$t = 1$}, \\
    \frac{15n-1}{56}-\frac{3}{5}\cdot\frac{1}{56}-\frac{7}{20(n-3)}& \hbox{$t = 2$}, \\
    \frac{15n-1}{56}-\frac{2}{3}\cdot\frac{1}{56}+\frac{7}{6(n+3)}& \hbox{$t = 3$}, \\
    \frac{15n-1}{56}-\frac{6}{5}\cdot\frac{1}{56}-\frac{7}{20(n-12)}& \hbox{$t = 4$}, \\
    \frac{15n-1}{56}-\frac{1}{3}\cdot\frac{1}{56}+\frac{7}{12(n+1)}& \hbox{$t = 5$}, \\
    \frac{15n-1}{56}-\frac{29}{27}\cdot\frac{1}{56}-\frac{35}{36(n-3)}& \hbox{$t = 6$}.
    \end{array}
\right.
\end{equation*}

\end{theorem}

Combining these bounds with Theorem \ref{THB}, we can obtain the
bounds on skew Randi\'c energy of trees.

We point out that the lower bound is sharp, and the equality in
lower bound holds  if and only if $T$ is a star with odd vertices
and an arbitrary orientation or $T=P_2$ with an arbitrary
orientation. If $n$ is odd, any oriented tree attains the lower
bound if and only if it is  a star and satisfies that its skew
Randi\'c spectrum is $\{i\mu_1,0,\cdots,0,-i\mu_1\}$, where
$\mu_1>0$. With Theorem \ref{th2.5}, we know that its Randi\'c
spectrum is $\{\mu_1,0,\cdots,0,-\mu_1\}$. Since 1 is the largest
Randi\'c eigenvalue of a connected graph, we have that $\mu_1=1$.
Then apply Theorem \ref{th0} to the odd ordered star $T$ with
$\lambda_2=0$, $\lambda_3=-1$, we can see that (i) and (ii)  hold.
So the skew Randi\'c energy of  a star with odd vertices and an
arbitrary orientation equals to $\sqrt{4 R_{-1}(T)}=2$ which reaches
the lower bound. If $n$ is even, any oriented tree attains the lower
bound must be a star and satisfy that $|\rho_i|=|\rho_j|$ for all
$i\neq j$. By Theorem \ref{th2.5}, that  implies the extremal trees
have exactly two distinct Randi\'c eigenvalues. But it is impossible
by Corollary \ref{t0} unless $T=P_2$. Thus, the equality in lower
bound holds  if and only if $T$ is a star with odd vertices and an
arbitrary orientation, or $T=P_2$.

Obviously, the upper bound on skew Randi\'c energies of trees that
obtained from Theorem \ref{THB}  and  Theorem \ref{t3} is not sharp.
From the proof of  Theorem \ref{THB}, we know that the oriented
trees attaining the upper bound in  Theorem \ref{THB} must satisfy
that $|\rho_i|=|\rho_j|$ for all $1\leq i\neq j\leq\lfloor
\frac{n}{2}\rfloor$. If $n$ is even, by  Theorem  \ref{th2.5}, the
extremal trees have at most two distinct Randi\'c eigenvalues, which
is impossible by   Corollary  \ref{t0}. If $n$ is odd, the Randi\'c
spectrum of $T$ is $\{1,\cdots,1,0,-1,\cdots,-1\}$. Since the
Randi\'c eigenvalue $1$ has multiplicity one for connected graphs,
we know that $T=P_3$, that is, $n=3$. Hence the upper bound on skew
Randi\'c energies of trees that comes from Theorem \ref{THB} and
Theorem  \ref{t3} cannot be sharp.
\\[2mm]

A \emph{chemical graph} is a graph in which no vertex has degree
greater than four. Analogously, a \emph{chemical tree} is a tree $T$
for which $\Delta(T ) \leq 4 $. Li and Yang \cite{LY2} gave the
sharp lower and upper bounds on $R_{-1}(T)$ among all chemical
trees.

\begin{theorem}\label{t4}
Let ${T}$
be a chemical tree of order $n$. Then,
\begin{equation*}
{R_{-1}(T )}\geq
\left\{
  \begin{array}{ll}
   1 & \hbox{if $n \leq 5$,} \\
    \frac{11}{8}& \hbox{if $n =6$}, \\
   \frac{3}{2}& \hbox{if $n =7$}, \\
    2& \hbox{if $n =10$}, \\
    \frac{3n+1}{16}& \hbox{for other values of $n$}. \\
    \end{array}
\right.
\end{equation*}

\end{theorem}

\begin{theorem}\label{t5}
Let ${T}$
be a chemical tree of order $n$, $n>6$. Then,
$${R_{-1}(T )}\leq max\{F_1(n),F_2(n),F_3(n)\},$$
where
\begin{equation*}
{F_1(n)}=
\left\{
  \begin{array}{ll}
   \frac{3n+1}{16}+ \frac{1}{144}\frac{31n+53}{3}& \hbox{if $n =1$ mod 3,} \\
    \frac{3n+1}{16}+ \frac{1}{144}\left(\frac{31n+22}{3}+9\right)& \hbox{if $n =2$ mod 3,} \\
   \frac{3n+1}{16}+ \frac{1}{144}\left(\frac{31n-9}{3}+18\right)& \hbox{if $n =0$ mod 3.} \\
    \end{array}
\right.
\end{equation*}
$$F_2(n)=\frac{3n+1}{16}+ \frac{1}{144}max\{11n - N_4 - 2k + 10, k = 0, 1, 2\}$$
with $N_4$ being the minimum integer solution of $n_4$ of the following system:
\begin{equation*}
\left\{
  \begin{array}{ll}
   n_3 + 2n_4 + 2 = n_1 \\
    2n_1 + n_3 + n_4 = n - k \\
   n_3 \leq 2n_4 + 2 \\
    \end{array}
\right.
\end{equation*}
and
$$F_3(n)=\frac{3n+1}{16}+ \frac{1}{144}max\{4n + 19N_1 + 5k + 4, k = 0, 1, 2\}$$
with $N_1$ being the minimum integer solution of $n_1$ of the following system:
\begin{equation*}
\left\{
  \begin{array}{ll}
  n_3 + 2n_4 + 2 = n_1 \\
    2n_1 + n_3 + n_4 = n - k \\
  n_3 \geq 2n_4 + 2 \\
  n_4 \geq 1.
    \end{array}
\right.
\end{equation*}
\end{theorem}

Combining these bounds with Theorem \ref{THB}, we can obtain the
bounds on skew Randi\'c energy of chemical trees. With the same
argument as before, we can get that the lower bound is sharp, and
the equality in lower bound holds  if and only if $T$ is a star with
2, 3 or 5 vertices and an arbitrary orientation. Also, the equality
in upper bound can not be attained by any chemical tree.

\section{Skew Randi\'c energies of trees}

It is well known that the skew energy of a directed tree is
independent of its orientation \cite{Adiga}.  In this section, we
investigate the  skew Randi\'c energy  of trees. Similarly, we
present a  basic lemma. The proof is  also similar to the proof
given in \cite{Adiga}.
\begin{lemma}\label{th3.1}
Let $D$ be a digraph, and let $D'$ be the digraph obtained from $D$
by reversing the orientations of all the arcs incident with a
particular vertex of $D$. Then $RE_s(D) = RE_s(D')$.\end{lemma}

\begin{proof}
Let $\textbf{R}_s(D)$ be the skew Randi\'c matrix of $D$ of order
$n$ with respect to a labeling of its vertex set. Suppose the
orientations of all the arcs incident at vertex $v_i$ of $D$ are
reversed. Let the resulting digraph be $D'$. Then $\textbf{R}_s(D')=
P_i\textbf{R}_s(D)P_i$ where $P_i$ is the diagonal matrix obtained
from the identity matrix of order $n$ by changing the $i-$th
diagonal entry to $-1$. Hence $\textbf{R}_s(D)$ and
$\textbf{R}_s(D')$ are orthogonally similar, and so have the same
eigenvalues, and hence $D$ and $D'$ have the same skew Randi\'c
energy.
\end{proof}

Let $\sigma$ be an orientation of a graph $G$. Let $W$ be a subset
of $V (G)$ and $\overline{W} = V (G)\setminus W$. The orientation
$\tau$ of $G$ obtained from  $\sigma$  by reversing the orientations
of all arcs between $W$ and $\overline{W}$ is said to be obtained
from $G^\sigma$ by a switching with respect to $W$. Moreover, two
oriented graphs $G^\tau$ and $G^\sigma$ of $G$ are said to be {\it
switching-equivalent} if $G^\tau$ can be obtained from $ G^\sigma $
by a switching. From lemma \ref{th3.1}, we know that
\begin{theorem}\label{th4.1}
If $G^\tau $ and $G^\sigma$ are switching-equivalent, then
$Sp_{\textbf{R}_s}(G^\sigma)=Sp_{\textbf{R}_s}(G^\tau).$
\end{theorem}

\begin{lemma}\cite{Adiga}\label{th3.2}
Let $T$ be a labeled directed tree rooted at vertex $v$. It is
possible, through reversing the orientations of all arcs incident at
some vertices other than $v$, to transform $T$ to a directed tree
$T'$ in which the orientations of all the arcs go from low labels to
high labels.
\end{lemma}
We can also  show that the skew energy of a directed tree is
independent of its orientation by using Lemma \ref{th3.1} and Lemma
\ref{th3.2}.
\begin{theorem}\label{th3.3}
The skew Randi\'c energy of a directed tree is independent of its
orientation.
\end{theorem}

\begin{coro}\label{cor3.4}
The skew Randi\'c energy of a directed tree is the same as the Randi\'c energy of its underlying tree.
\end{coro}
We omit the proofs of Theorem \ref{th3.3} and Corollary
\ref{cor3.4}, since they are similar to the proofs of Theorem 3.3
and Corollary 3.4 in \cite{Adiga}.

\section{Graphs with $Sp_{\textbf{R}_s}(G^\sigma)=\mathbf{i}Sp_{\textbf{R}}(G)$}

The relationship between $Sp_s(G^\sigma)$ and $\mathbf{i}Sp(G)$ has
been concerned in \cite{SS}. Similarly, we concentrate on the
relationship between $Sp_{\textbf{R}_s}(G^\sigma)$ and
$\mathbf{i}Sp_{\textbf{R}}(G)$, and we obtain some analogous
results. The following two lemmas given in \cite{SS} will be used.
\begin{lemma}\cite{SS}\label{lem2.1}
Let $\mathbf{A}=\left( {\begin{array}{*{20}{c}}
\mathbf{0}&X\\
{{X^T}}&\mathbf{0}
\end{array}} \right)$  and $\mathbf{B}=\left( {\begin{array}{*{20}{c}}
\mathbf{0}&X\\
{{-X^T}}&\mathbf{0}
\end{array}} \right)$ be two real matrices. Then
$Sp(\mathbf{B}) = \mathbf{i}Sp(\mathbf{A})$.
\end{lemma}

Let $|\mathbf{X}|$ denote the matrix whose entries are the absolute
values of the corresponding entries in $\mathbf{X}$. For real
matrices $\mathbf{X}$ and $\mathbf{Y}$, $\mathbf{X} \le \mathbf{Y}$
means that $\mathbf{Y} - \mathbf{X}$ has nonnegative entries.
$\rho(\mathbf{X})$ denotes the spectral radius of a square matrix
$\mathbf{X}$.
\begin{lemma}\cite{SS}\label{lem2.3}
Let $ \mathbf{A}$ be an irreducible nonnegative matrix and
$\mathbf{B}$ be a real positive semi-definite matrix such that
$|\mathbf{B}| \le \mathbf{A}$ (entry-wise) and $\rho(\mathbf{A})
=\rho(\mathbf{B})$. Then $\mathbf{A} = \mathbf{DBD}$ for some real
matrix $\mathbf{D}$ such that $|\mathbf{D}| = \mathbf{I}$, the
identity matrix.
\end{lemma}

\begin{theorem}\label{th2.2}
$G$ is a bipartite graph if and only if there is an orientation
$\sigma$ such that
$Sp_{\textbf{R}_s}(G^\sigma)=\mathbf{i}Sp_{\textbf{R}}(G)$.
\end{theorem}

\begin{proof}
(Necessity) If $G$ is bipartite, then there is a labeling such that
the Randi\'c matrix of $G$ is of the form
$$\mathbf{R}(G)=\left( {\begin{array}{*{20}{c}}
\mathbf{0}&\mathbf{X}\\
{{\mathbf{X}^T}}&\mathbf{0}
\end{array}} \right).$$
Let $\sigma$ be the orientation such that the skew Randi\'c matrix
of $G^\sigma$ is of the form
$${\textbf{R}_s}(G^\sigma)=\left( {\begin{array}{*{20}{c}}
\mathbf{0}&\mathbf{X}\\
{-{\mathbf{X}^T}}&\mathbf{0}
\end{array}} \right).$$
By Lemma \ref{lem2.1}, $Sp_{\textbf{R}_s}(G^\sigma)=\mathbf{i}Sp_{\textbf{R}}(G)$.

(Sufficiency) Suppose that
$Sp_{\textbf{R}_s}(G^\sigma)=\mathbf{i}Sp_{\textbf{R}}(G)$, for some
orientation $\sigma$. Since ${\textbf{R}_s}(G^\sigma)$ is a real
skew symmetric matrix, $Sp_{\textbf{R}_s}(G^\sigma)$ has only pure
imaginary eigenvalues and so is symmetric about the real axis. Then
$Sp_{\textbf{R}}(G)=-\mathbf{i}Sp_{\textbf{R}_s}(G^\sigma)$ is
symmetric about the imaginary axis. Hence $G$ is bipartite.
 \end{proof}

Trees are special bipartite graphs, actually, we will prove that $G$ is a tree
if and only if for any orientation $\sigma$,
$Sp_{\textbf{R}_s}(G^\sigma)=\mathbf{i}Sp_{\textbf{R}}(G)$.
The next  lemma plays an important role in the proof of the above statement.

\begin{lemma}\label{lem2.4}
Let $ \mathbf{X}=\left( {\begin{array}{*{20}{c}}
\mathbf{C}&*\\
{{*}}&*
\end{array}} \right)$ be a nonnegative matrix, where
$\mathbf{C}=(c_{ij})$ is a $k \times k$ $(k > 2)$ matrix whose
nonzero entries are $c_{i,i-1}$ and $c_{i,i}$ with the subscripts
modulo $k$,  for $1\leq i\leq k$. Let $\mathbf{Y}$ be obtained from
$ \mathbf{X}$ by changing the $(1,1)$ entry to $-c_{1,1}$. If
$\mathbf{X}^T\mathbf{X}$ is irreducible then
$\rho(\mathbf{X}^T\mathbf{X}) >\rho(\mathbf{Y}^T\mathbf{Y})$.
\end{lemma}

\begin{proof}
Note that $|\mathbf{Y}^T\mathbf{Y}| \le \mathbf{X}^T\mathbf{X}$
(entry-wise), and so $\rho(\mathbf{X}^T\mathbf{X})
\ge\rho(\mathbf{Y}^T\mathbf{Y})$ by Perron- Frobenius theory
\cite{HJ}. Now suppose that $\rho(\mathbf{X}^T\mathbf{X})
=\rho(\mathbf{Y}^T\mathbf{Y})$. Since $\mathbf{X}^T\mathbf{X} $ is
irreducible, by Lemma \ref{lem2.3}, there exists a signature matrix
$\mathbf{D} = Diag(d_1, d_2, \ldots , d_n)$ such that
$\mathbf{X}^T\mathbf{X}= \mathbf{D}\mathbf{Y}^T\mathbf{Y}\mathbf{D}$
. Therefore $[\mathbf{X}^T\mathbf{X}]_{ij} = d_id_j
[\mathbf{Y}^T\mathbf{Y}]_{ij }$ for all $i$,  $j$.  Now, for $i = 1,
\cdots , k - 1$,
$[\mathbf{X}^T\mathbf{X}]_{i,i+1}=[\mathbf{Y}^T\mathbf{Y}]_{i,i+1}\neq
0$. Using $[\mathbf{X}^T\mathbf{X}]_{ij} = d_id_j
[\mathbf{Y}^T\mathbf{Y}]_{ij }$ , we have $d_id_{i+1} = 1$ for $i =
1, \cdots , k-1$. Hence $d_1d_k = 1$. On the other hand, let
$[\mathbf{X}^T\mathbf{X}]_{1k} =c_{1,1}c_{1,k}+M$, and so we have
$-c_{1,1} c_{1,k}+M=d_1d_k[\mathbf{Y}^T\mathbf{Y}]_{1k
}=[\mathbf{X}^T\mathbf{X}]_{1k} =c_{1,1} c_{1,k}+M$, which is
impossible.
\end{proof}

\begin{theorem}\label{th2.5}
 $G$ is a tree
if and only if for any orientation $\sigma$,
$Sp_{\textbf{R}_s}(G^\sigma)=\mathbf{i}Sp_{\textbf{R}}(G)$.
\end{theorem}
\begin{proof}
(Necessity) The necessity follows directly from Theorem \ref{th3.3}
and Theorem \ref{th2.2}.

(Sufficiency) Suppose that
$Sp_{\textbf{R}_s}(G^\sigma)=\mathbf{i}Sp_{\textbf{R}}(G)$, for any
orientation $\sigma$. By Theorem \ref{th2.2}, $G$ is a bipartite
graph. So there is a labeling such that the Randi\'c matrix of $G$
is of the form
$$\mathbf{R}(G)=\left( {\begin{array}{*{20}{c}}
\mathbf{0}&\mathbf{X}\\
{{\mathbf{X}^T}}&\mathbf{0}
\end{array}} \right),$$
where $\mathbf{X}$ is a nonnegative matrix. Since $G$ is connected, $\mathbf{X}^T\mathbf{X}$ is indeed
a positive matrix and so irreducible. Now assume that $G$ is not a tree. Then $G$
has at least an even cycle because $G$ is bipartite. W.L.O.G. $\mathbf{X}$ has the form $\left( {\begin{array}{*{20}{c}}
\mathbf{C}&*\\
{{*}}&*
\end{array}} \right)$
where $\mathbf{C}=(c_{ij})$ is a $k \times k$ $(k > 2)$ matrix whose
nonzero entries are $c_{i,i-1}$ and $c_{i,i}$ with the subscripts
modulo $k$,  for $1\leq i\leq k$. Let $\mathbf{Y}$ be obtained from
$ \mathbf{X}$ by changing the $(1,1)$ entry to $-c_{1,1}$.  Consider
the orientation $\sigma$ of $G$ such that
$${\textbf{R}_s}(G^\sigma)=\left( {\begin{array}{*{20}{c}}
\mathbf{0}&\mathbf{Y}\\
{-{\mathbf{Y}^T}}&\mathbf{0}
\end{array}} \right).$$

By hypothesis,
$Sp_{\textbf{R}_s}(G^\sigma)=\mathbf{i}Sp_{\textbf{R}}(G)$
and hence $\mathbf{X}$ and $\mathbf{Y}$ have the same singular values.
It follows that $\rho(\mathbf{X}^T\mathbf{X})=\rho(\mathbf{Y}^T\mathbf{Y})$,
which contradicts Lemma \ref{lem2.4}.
 \end{proof}

\begin{coro}\label{cor1}
$G$ is a forest
if and only if for any orientation $\sigma$,
$Sp_{\textbf{R}_s}(G^\sigma)=\mathbf{i}Sp_{\textbf{R}}(G)$.
 \end{coro}

\begin{proof}
(Necessity) Let $G = G_1\cup \ldots \cup G_r$ where $G_j$'s are
trees. Then $G^\sigma = G_1^{\sigma_1}\cup \ldots \cup
G_r^{\sigma_r}$. By Theorem \ref{th2.5},
$Sp_{\textbf{R}_s}(G_j^{\sigma_j})=\mathbf{i}Sp_{\textbf{R}}(G_j)$
for all $j = 1, 2, \cdots , r$. Hence
$Sp_{\textbf{R}_s}(G^\sigma)=Sp_{\textbf{R}_s}(G_1^{\sigma_1}) \cup
\cdots\cup Sp_{\textbf{R}_s}(G_r^{\sigma_r})=
\mathbf{i}Sp_{\textbf{R}}(G_1)\cup \cdots\cup
\mathbf{i}Sp_{\textbf{R}}(G_r) = \mathbf{i}Sp_{\textbf{R}}(G_1\cup
\ldots \cup G_r) = \mathbf{i}Sp_{\textbf{R}}(G)$.

(Sufficiency) Suppose that $G$ is not a forest. Then $G = G_1\cup
\ldots \cup G_r$ where $G_1\ldots  G_t$ are connected, but not
trees, and $G_{t+1}\ldots  G_r$ are trees. By Theorem \ref{th2.2},
$G$ is a bipartite graph.  So there is a labeling such that the
Randi\'c matrix of $G$ is of the form
 $$\mathbf{R}(G)=\left( {\begin{array}{*{20}{c}}
\mathbf{0}&\mathbf{X}\\
{{\mathbf{X}^T}}&\mathbf{0}
\end{array}} \right),$$
where $\mathbf{X}=\mathbf{X_1}\bigoplus\cdots\bigoplus \mathbf{X_r}$.
Let $\mathbf{Y_j}$ be obtained from $ \mathbf{X_j}$ by changing
the $(1,1)$ entry to its negative.  Consider the orientation
$\sigma$ of $G$ such that
$${\textbf{R}_s}(G^\sigma)=\left( {\begin{array}{*{20}{c}}
\mathbf{0}&\mathbf{Y}\\
{-{\mathbf{Y}^T}}&\mathbf{0}
\end{array}} \right).$$
where $\mathbf{Y}=\mathbf{Y_1}\bigoplus\cdots\bigoplus
\mathbf{Y_r}$. By Lemma \ref{lem2.1},
$Sp_{\textbf{R}_s}(G^\sigma)=\mathbf{i}Sp_{\textbf{R}}(G)$ implies
that the singular values of $\mathbf{X}$ coincide with the singular
values of $\mathbf{Y}$. Since $G_{t+1}\ldots  G_r$ are trees, the
singular values of $\mathbf{X_j}$ coincide with the singular values
of $\mathbf{Y_j}$ for $j = t + 1,\cdots , r$. Hence the singular
values of $\mathbf{X_1}\bigoplus\cdots\bigoplus \mathbf{X_t}$
coincide with the singular values of
$\mathbf{Y_1}\bigoplus\cdots\bigoplus \mathbf{Y_t}$. Since
$G_1\ldots  G_t$ are not trees, we have
$\rho(\mathbf{X_j}^T\mathbf{X_j})>\rho(\mathbf{Y_j}^T\mathbf{Y_j})$
for $j = 1,\cdots, t$. Consequently, for some $j_0$,
$$\mathop {\max }\limits_{1 \le j \le t}\rho(\mathbf{X_j}^T\mathbf{X_j})=\mathop {\max }\limits_{1 \le j \le t} \rho(\mathbf{Y_j}^T\mathbf{Y_j})=
\rho(\mathbf{Y_{j_0}}^T\mathbf{Y_{j_0}})<\rho(\mathbf{X_{j_0}}^T\mathbf{X_{j_0}})\le\mathop {\max }\limits_{1 \le j \le t}\rho(\mathbf{X_j}^T\mathbf{X_j}),$$
a contradiction.
 \end{proof}

Let $\sigma$ be an orientation of $G$. From Theorem \ref{th2.2}, we
know that $Sp_{\textbf{R}_s}(G^\sigma)=\mathbf{i}Sp_{\textbf{R}}(G)$
only if $G$ is bipartite. We concentrate on the orientation $\sigma$
of bipartite graph $G$ so that
$Sp_{\textbf{R}_s}(G^\sigma)=\mathbf{i}Sp_{\textbf{R}}(G)$ in the
sequel. Let the characteristic polynomials of $\textbf{R}(G)$ and
$\textbf{R}_s(G^\sigma)$  be expressed as in \eqref{1} and
\eqref{2}, respectively.
$Sp_{\textbf{R}_s}(G^\sigma)=\mathbf{i}Sp_{\textbf{R}}(G)$ if and
only if
$\varphi_R(G,\lambda)=\sum_{i=0}^na_i\lambda^{n-i}=\lambda^{n-2r}\prod_{i=1}^r
(\lambda^2- \lambda_i^2)$ and
$\varphi_{R_s}(G^\sigma,\lambda)=\sum_{i=0}^n
c_i\lambda^{n-i}=\lambda^{n-2r}\prod_{i=1}^r
(\lambda^2+\lambda_i^2)$ for some non-zero real numbers
$\lambda_1,\lambda_2,\cdots, \lambda_r$. Hence, we have that
$Sp_{\textbf{R}_s}(G^\sigma)=\mathbf{i}Sp_{\textbf{R}}(G)$  if and
only if
\begin{equation}\label{eq1}
  a_{2i}=(-1)^ic_{2i}, \  a_{2i+1}=c_{2i+1}=0,
\end{equation}
for $i=0,1,\cdots,\lfloor\frac{n}{2}\rfloor.$

An even cycle $C_{2l}$ is said to be \emph{oriented uniformly} if
$C_{2l}$ is oddly (resp.,evenly) oriented relative to $G^\sigma$
when $l$ is odd (resp., even). If every even cycle in $G^\sigma$  is
oriented uniformly, then the orientation $\sigma$ is defined to be a
\emph{parity-linked} orientation of $G$.

\begin{theorem}\label{th6.1}
Let $G$ be a bipartite graph and $\sigma$ be an orientation of $G$.
Then $Sp_{\textbf{R}_s}(G^\sigma)=\mathbf{i}Sp_{\textbf{R}}(G)$  if
and only if $\sigma$ is  a parity-linked orientation of $G$.
 \end{theorem}
\begin{proof} Since $G$ is bipartite, all cycles in $G$ are even
and all linear subgraphs are even. Then $a_{2i+1}= 0$ for all $i$.

(Sufficiency) Since every even cycle is oriented uniformly, for
every cycle $C_{2l}$ with length $2l$, $C_{2l}$ is evenly oriented
relative to $G^\sigma$ if and only if $l$ is even. Thus
$(-1)^{p_e(C_{2l})} = (-1)^{l+1}$.

By Eqs \eqref{ai} and \eqref{ci}, we have
\begin{eqnarray*}
  (-1)^ia_{2i} &=& \sum_{L\in \mathcal{M}_{2i}}W(L) +\sum_{L\in \mathcal{CL}_{2i}}(-1)^{p_1(L)+i}(-2)^{p_2(L)}W(L),\\
  c_{2i} &=& \sum_{L\in \mathcal{M}_{2i}}W(L)+\sum_{L\in \mathcal{CL}_{2i}}(-2)^{p_e(L)}2^{p_o(L)}W(L),
\end{eqnarray*}
where $\mathcal{M}_{2i}$ is the set of matchings with $i$ edges and
$\mathcal{CL}_{2i}$ is the set of all linear subgraphs with $2i$
vertices of $G$ and containing at least one cycle.

For a linear subgraph $L \in \mathcal{CL}_{2i}$ of $G$, assume that
$L$ contains the cycles $C_{2l_1},C_{2l_2},\cdots,C_{2l_{p_2}}$,
then the number of components of $L$ that are single edges is
$p_1(L)=i-\sum_{j=1}^{p_2(L)}l_j$. Hence
$(-1)^{p_1(L)+i}=(-1)^{\sum_{j=1}^{p_2(L)}l_j}$.

Therefore, for a linear subgraph  $L\in \mathcal{CL}_{2i}$,
$$(-2)^{p_e(L)}2^{p_o(L)}=(-1)^{p_e(L)}2^{p_2(L)}=
(-1)^{\sum_{j=1}^{p_2}(l_j+1)}2^{p_2(L)}=(-1)^{p_1(L)+i}(-2)^{p_2(L)}.$$
 Thus $(-1)^ia_{2i}=c_{2i}$, and the sufficiency is proved.

(Necessity) If there is an even cycle of $G$ that is not oriented
uniformly in $G^\sigma$, then choose a shortest cycle $C_{2l}$ with
length $2l$ such that $C_{2l}$ is not oriented uniformly, that is,
$C_{2l}$ is oddly oriented in $G^\sigma$  if $l$ is even, and evenly
oriented if $l$ is odd. Let $\mathcal{C}_{2l}$ be  the set of cycles
with length $2l$ such that they are not oriented uniformly, and let
$\mathcal{UCL}_{2l}$ denote the set of all even linear subgraphs
with $2l$ vertices of $G$ and all even cycles that are oriented
uniformly. Thus, we have

\begin{eqnarray*}
(-1)^la_{2l} &=& \sum_{L\in \mathcal{M}_{2l}}W(L)+\sum_{L\in \mathcal{C}_{2l}}(-1)^l(-2)W(L)+\sum_{L\in \mathcal{UCL}_{2l}}(-1)^{p_1(L)+l}(-2)^{p_2(L)}W(L),\\
c_{2l} &=& \sum_{L\in \mathcal{M}_{2l}}W(L)+\sum_{L\in
\mathcal{C}_{2l}}(-1)^l2W(L)+\sum_{L\in
\mathcal{UCL}_{2l}}(-2)^{p_e(L)}2^{p_o(L)}W(L).
\end{eqnarray*}
By the choice of $C_{2l}$ and the proof of the necessity, we have that $$\sum_{L\in \mathcal{UCL}_{2l}}(-1)^{p_1(L)+l}(-2)^{p_2(L)}W(L)=\sum_{L\in \mathcal{UCL}_{2l}}(-2)^{p_e(L)}2^{p_o(L)}W(L).$$
However, $$\sum_{L\in \mathcal{C}_{2l}}(-1)^l(-2)W(L)\neq \sum_{L\in \mathcal{C}_{2l}}(-1)^l2W(L).$$
Thus $(-1)^la_{2l}\neq c_{2l}$. This contradicts $Sp_{\textbf{R}_s}(G^\sigma)=\mathbf{i}Sp_{\textbf{R}}(G)$ by Eq. \eqref{eq1}.
\end{proof}

Let $G$ be a bipartite graph with the bipartition $V (G) = X \cup
Y$. We call an orientation $\sigma$ of $G$ the \emph{canonical}
orientation if it assigns each edge of $G$ the direction from $X$ to
$Y$.
 For the canonical orientation $\sigma$ of $G(X, Y )$, from the proof of Theorem \ref{th2.2}, we have that
\begin{equation}\label{eq3}
Sp_{\textbf{R}_s}(G^\sigma)=\mathbf{i}Sp_{\textbf{R}}(G).
\end{equation}

\begin{theorem} Suppose $\tau$ is an orientation of a bipartite graph $G =
G(X, Y )$. Then
$Sp_{\textbf{R}_s}(G^\tau)=\mathbf{i}Sp_{\textbf{R}}(G)$ if and only
if $G^\tau$ is switching-equivalent to $G^\sigma$, where $\sigma$ is
the canonical orientation of $G$.
\end{theorem}
\begin{proof}
The sufficiency can be easily obtained from Theorem \ref{th4.1} and Eq. \eqref{eq3}.

We prove the necessity  by induction on the number of edges $m$ of
the bipartite graph $G$ in the following. The result is trivial for
$m = 1$. Assume that the result is true for all bipartite graphs
with at most $m - 1$ $(m >2)$ arcs. Let $G$ be a bipartite graph
with $m$ edges and $(X, Y )$ be the bipartition of the vertex set of
$G$. Suppose that $\tau$ is an orientation of $G$ such that
$Sp_{\textbf{R}_s}(G^\tau)=\mathbf{i}Sp_{\textbf{R}}(G)$. We have to
prove that $\tau$ is switching-equivalent to $\sigma$. Let $e$ be
any edge of $G$. By Theorem \ref{th6.1} every even cycle is oriented
uniformly in $G^{\tau}$ and hence in $(G-e)^{\tau}$. Consequently,
$\tau_e$ is a parity-linked orientation of $G-e$, where $\tau_e$ is
the restriction of $ \tau$ to the graph $G-e$. So by Theorem
\ref{th6.1},
$Sp_{\textbf{R}_s}((G-e)^{\tau_e})=\mathbf{i}Sp_{\textbf{R}}(G-e)$.

Consequently, by induction hypothesis, $(G-e)^{\tau_e}$ is
switching-equivalent to $(G-e)^{\sigma_e}$, where $\sigma_e$ is the
restriction of $\sigma$ to the graph $G-e$. Let $\alpha$ be the
switch that takes $(G-e)^{\tau_e}$ to $(G-e)^{\sigma_e}$ effected by
the subset $U$ of $V(G-e)=V(G)$. We claim that $\alpha$ takes $\tau$
to $\sigma$ in $G$. If not, then the resulting oriented graph
$G^{\tau'}$ will be of the following type (see Fig.1): all the arcs
of $G-e$ will be oriented from one partite set (say, $X$) to the
other (namely, $Y$ ) while the arc $e$ will be oriented in the
reverse direction, that is, from $Y$ to $X$.
\begin{figure}[h,t,b,p]
\begin{center}
\scalebox{0.8}[0.8]{\includegraphics{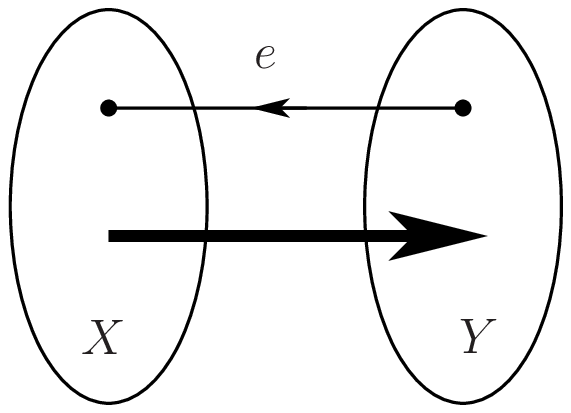}}\\[12pt]
Figure~1. the  oriented graph $G^{\tau'}$
\end{center}
\end{figure}

Consider first the case when $e$ is a cut edge of $G$. The subgraph
$G-e$ will then consist of two components with vertex sets, say,
$S_1$ and $S_2$. Now switch with respect to $S_1$. This will change
the orientation of the only arc $e$ and the resulting orientation is
$\sigma$. Consequently, $\tau$ is switching-equivalent to $\sigma$.

Note that the above argument also takes care of the case when $G$ is
a tree since each edge of $G$ will then be a cut edge. Hence we now
assume that $G$ contains an even cycle $C_{2k}$ containing the arc
$e$. But then any such $C_{2k}$ has $k-1$ arcs in one direction and
$k+1$ arcs in the opposite direction, thereby not oriented
uniformly. Hence this case can not arise. Consequently, $\tau$ is
switching-equivalent to $\sigma$  in $G$.
\end{proof}


\begin{thebibliography}{10}

\bibitem{Adiga} C. Adiga, R. Balakrishnan, Wasin So, The skew
energy of a digraph, {\it Linear Algebra  Appl.} {\bf432} (2010), 1825--1835.



\bibitem{BE}
B. Bollob\'{a}s, P. Erd\H{o}s, Graphs of extremal weights, {\it Ars Combin.}  {\bf 50} (1998), 225--233.

\bibitem{Bondy}
J.A. Bondy,  U.S.R. Murty,  {\it Graph Theory}, Springer, New York,
2008.

\bibitem{CM}
L.H. Clark, J.W. Moon, On the general Randi\' c index for certain families of trees, {\it Ars Combin.}  {\bf 54} (2000), 223--235.

\bibitem{DM}
D.M. Cvetkovi\'{c}, M. Doob, H. Sachs, {\it Spectra of Graphs--Theory and Application},
Academic Press, New York, 1980.



\bibitem{Gutman}
 I. Gutman, B. Furtula, \c{S}.B. Bozkurt, On Randi\' c energy,  {\it Linear
Algebra Appl.} {\bf442} (2014), 50--57.


\bibitem{GHL}
R. Gu, F. Huang, X. Li,
General Randi\' c matrix and general Randi\' c energy,
{\it Trans. Combin.} {\bf3} (2014) 21--33.


\bibitem{GHL2}
R. Gu, F. Huang, X. Li,
Randi\' c incidence energy of graphs,
{\it Trans. Combin.} {\bf3}
(2014) 1--9.

\bibitem{GX}
 S. Gong, G. Xu, The characteristic polynomial and the matchings polynomial of a weighted oriented graph,
{\it Linear
Algebra Appl.} {\bf436} (2012), 3597--3607.

\bibitem{HJ}
R. Horn, C. Johnson, {\it Matrix Analysis}, Cambridge University Press, 1987.

\bibitem{LW}
X. Li, J. Wang,  Randi\'c energy and Randi\'c eigenvalues, {\it MATCH Commun.
Math. Comput. Chem.} {\bf 73} (2015), 73--80.


\bibitem{LY}
X. Li, Y. Yang, Sharp bounds for the general Randi\'c index, {\it MATCH Commun.
Math. Comput. Chem.} {\bf 51} (2004), 155--166.

\bibitem{LY2}
X. Li, Y. Yang, Best lower and upper bounds for the Randi\'c index $R_{-1}$ of chemical trees, {\it MATCH Commun.
Math. Comput. Chem.} {\bf 52} (2004), 147--156.

\bibitem{PSL}
Lj. Pavlovi\'c, M. Stojanvoi\'c, X. Li, More on the best upper bound for the Randi\'c index $R_{-1}$ of
trees, {\it MATCH Commun.
Math. Comput. Chem.} {\bf 60} (2008), 567--584.


\bibitem{Randic}
M. Randi\'c, On characterization of molecular branching, {\it  J.
Am. Chem. Soc.}  {\bf 97} (1975), 660--6615.

\bibitem{SS}
B. Shader, Wasin So, Skew spectra of oriented graphs, {\it   The
Electron. J. Combin.} {\bf 16} (2009), \#N32.


\end{thebibliography}
\end{document}